\documentclass{amsart}

\usepackage{amssymb,amsmath,amsthm,amsfonts}
\usepackage[mathscr]{euscript}

\renewcommand{\mathcal}{\mathscr}
\renewcommand{\leq}{\leqslant}
\renewcommand{\le}{\leqslant}
\renewcommand{\geq}{\geqslant}
\renewcommand{\ge}{\geqslant}

\usepackage{color}
\definecolor{citation}{rgb}{0.2,0.5,0.2}
\definecolor{formula}{rgb}{0.1,0.2,0.5}
\definecolor{url}{rgb}{0,0.2,0.7}
\usepackage[colorlinks=true,linkcolor=formula,urlcolor=url,citecolor=citation]{hyperref}

\newtheorem{theorem}{Theorem}[section]
\newtheorem{corollary}[theorem]{Corollary}
\newtheorem{lemma}[theorem]{Lemma}
\newtheorem{prop}[theorem]{Proposition}

\theoremstyle{definition}
\newtheorem{defn}[theorem]{Definition}
\theoremstyle{remark}
\newtheorem{rem}[theorem]{Remark}
\numberwithin{equation}{section}

\def\R {\mathbb{R}}

\def\N {\mathbb{N}}
\def\S {\mathbb{S}}

\newlength{\defbaselineskip}
\newcommand{\setlinespacing}[1]
           {\setlength{\baselineskip}{#1 \defbaselineskip}}

\begin{document}

\title[Some monotonicity results for general systems of PDEs]
{Some monotonicity results \\
for general systems of nonlinear elliptic PDEs}

\author[Julien Brasseur]{Julien Brasseur}
\address[Julien Brasseur]{Institut de Math\'ematiques de Marseille (I2M),
Centre de Math\'ematiques et Informatique (CMI),
Technop\^ole Ch\^ateau-Gombert,
39, rue F. Joliot Curie,
13453 Marseille Cedex 13, France}
\email{julien.brasseur@ens-lyon.fr}

\author[Serena Dipierro]{Serena Dipierro}
\address[Serena Dipierro]{Institut f\"ur Analysis und Numerik,
Otto-von-Guericke Universit\"at Magdeburg,
Universit\"atsplatz 2, 39106 Magdeburg, Germany}
\email{serydipierro@yahoo.it}

\begin{abstract}
In this paper we show that minima and stable solutions
of a general energy functional of the form
$$ \int_{\Omega} F(\nabla u,\nabla v,u,v,x)dx $$
enjoy some monotonicity properties,
under an assumption on the growth at infinity of the energy.

Our results are quite general, and comprise some rigidity results which are known
in the literature.
\end{abstract}



\maketitle

\tableofcontents

\section{Introduction}

\subsection{Monotonicy and $1$-dimensional symmetry for systems}

In this paper we study monotonicity properties of minima and stable solutions
of a quite general energy functional of the type
\begin{equation}\label{quitegen}
\int_{\Omega} F(\nabla u,\nabla v,u,v,x)dx,
\end{equation}
where $\Omega\subset\R^n$.

Recent years have seen numerous ongoing research activities
in investigating symmetry properties of systems of PDEs. A typical example is
\begin{eqnarray}
\left\{
\begin{array}{r c l}
\Delta u=uv^2,  \\
\Delta v=vu^2,\\
u,v>0.
\end{array}
\right. \label{bose}
\end{eqnarray}
which arise in phase separation for Bose-Einstein condensates with multiple states.
We refer to \cite{BLWZ, BSWW} and references therein
for a derivation of this phase separation model.

In particular, in \cite{BLWZ}, the authors prove existence,
symmetry and non-degeneracy for solutions to \eqref{bose} in 1-dimension.
They show that entire solutions are reflectionally symmetric,
i.e. there exists $x_0\in\R$ such that $u(x-x_0)=v(x_0-x)$ for any $x\in\R$.
They also establish a result which may be seen as the analogue of a celebrated
conjecture of De Giorgi for problem \eqref{bose} in dimension $2$;
more precisely, they show that monotone solutions to \eqref{bose} in dimension $2$
have one-dimensional symmetry, under the growth condition
$$ u(x)+v(x)\le C(1+|x|), $$
for some $C>0$.

On the other hand, the linear growth is the lowest possible for positive solutions to \eqref{bose}.
Namely, if there exists $\alpha\in(0,1)$ such that
$$ u(x)+v(x)\le C(1+|x|)^\alpha,$$
then $u=v\equiv0$, see \cite{noris}.

In dimension $2$, in \cite{farina} improves the result in \cite{BLWZ}
showing that if $(u,v)$ is a monotone solution
to \eqref{bose} and has at most algebraic growth at infinity then it must be one-dimensional.

The monotonicity condition can be weakened in order to get the one-dimensional symmetry for solutions to \eqref{bose}. Indeed, it has been proved in \cite{BSWW} stable solutions to \eqref{bose}
are one-dimensional in $\R^2$.

Symmetry results in dimension $2$ for systems with more general nonlinearities have been achieved in
\cite{fazly}. Quasilinear (possibly degenerate) elliptic systems in $\R^2$ have
also been considered in \cite{Di}.

As for dimensions higher than $2$, in \cite{soave} the authors prove that
if $(u,v)$ has algebraic growth and
$$ \lim_{x_n\to\pm\infty}\left(u(x',x_n)-v(x',x_n)\right) =\pm\infty \quad
{\mbox{ uniformly in }} x'\in\R^{n-1},$$
then $(u,v)$ only depends on $x_n$. Also, in \cite{Wa,Wa2} it has been proved that
if $(u,v)$ has linear growth then it is one-dimensional.

Recently the nonlocal counterpart of \eqref{bose} has been also investigated,
see e.g. \cite{zilio, zilio1, zilio2}, and some symmetry results have been obtained in
\cite{pinamonti} for a quite general system of nonlocal PDEs of the form
\begin{eqnarray*}
\left\{
\begin{array}{r c l}
(-\Delta)^{s_1} u=F_1(u,v),  \\
(-\Delta)^{s_2} v=F_2(u,v),
\end{array}
\right.
\end{eqnarray*}
where $F_1$ and $F_2$ denote the derivatives of a function $F\in C^{1,1}_{loc}(\R^2)$
with respect to the first and the second variable, respectively; $s_1$, $s_2\in(0,1)$ and, given $s\in(0,1)$, $(-\Delta)^s$ stands for the fractional Laplace operator
$$ (-\Delta)^su(x):= \mathrm{p.v.}\,\int_{\R^n}\frac{u(x)-u(y)}{|x-y|^{n+2s}}\,dy,$$
where p.v. denotes the Cauchy principal value
(see \cite{guida} for the definition and further details).
See also \cite{sire} for symmetry results for nonlocal systems of equations.

\medskip

Aim of this paper is to provide monotonicity results for minima and stable solutions
of energy functionals of the form \eqref{quitegen} (thus embracing all the systems considered above).
We will consider the case in which both the domain $\Omega$ and the functional $F$ are
invariant under translations in the $e_n$-direction, and we will
deal with symmetry properties of minima or stable solutions in the class of functions
which are obtained by piecewise Lipschitz domain deformations in the $e_n$-direction
(see Definitions \ref{deform} and \ref{pert}).
We will also require some mild assumptions on $F$ and a growth condition (see \eqref{growth}).

The key tool of our proofs rely on a technique introduced
in \cite{SaVa2} to study the regularity of fractional minimal surfaces in dimension 2
(see \cite{CRS} where these object were introduced),
and developed in \cite{SaVa} to work in a more general setting.
See also \cite{CSaVa, DipSV} where these techniques have been used
in the context of free boundary problems.

Here we adapt the new strategy of \cite{SaVa} to the case of general systems of equations.
The idea is to look at the stability inequality without dealing with its precise form
(which, in some cases, can be very hard to handle), and
simply compare the energies of the couple $(u,v)$ to that of its translations.
For this, one also needs to modify the ``translated'' couple at infinity
to make it a compact perturbation of $(u,v)$.
Here the growth condition comes into play and ensures that the
energy of the perturbed couple can be made arbitrarily close
to the energy of $(u,v)$. Moreover, if $u$ and $v$ are not monotone
in the $e_n$-direction, then one can modify locally the perturbed couple
in order to get lower energy, but this is in contradiction with
the minimality of $(u,v)$.

This strategy allows to deal with quite general form of energy functionals.
\medskip

We now introduce the setting in which we will work
and give precise definitions and statements
of our results.

\subsection{The mathematical setting}

We consider a domain~$\Omega\subset\R^n$ and study the symmetry properties
of minima for functionals of the type
$$ \mathcal{E}(u,v):=\int_{\Omega} F(\nabla u, \nabla v, u, v, x)dx. $$
We assume that the domain~$\Omega$ and the functional~$F$ are invariant under
translations in the~$e_n$-direction, namely
$$ \Omega= V \times\R, \qquad V \subseteq\R^{n-1}, $$
and~$F$ does not depend on the~$x_n$-coordinate.

We will denote a point in~$\Omega$ by~$x=(x',x_n)\in V\times\R$.

We also suppose that~$F$ is convex with respect to the first two variables.
Precisely, we assume that
\begin{equation}\label{F1}
F=F(p_1,p_2,z_1,z_2,x')\in C(\R^{2n}\times\R^2\times V),
\end{equation}
and~$F$ is~$C^2$ and uniformly convex in~$p_1$ and~$p_2$ at all~$p_1,p_2$
with~$p_1\cdot e_n\neq0$ and~$p_2\cdot e_n\neq0$, for all~$(z_1,z_2,x')\in\R^2\times V$.
Finally, we assume that there exists a constant~$C>0$ such that
\begin{equation}\begin{split}\label{F2}
& |F_{p_1p_1}(p_1+q_1,p_2,z_1,z_2,x')|\leq C|F_{p_1p_1}(p_1,p_2,z_1,z_2,x')|, \\
& |F_{p_2p_2}(p_1,p_2+q_2,z_1,z_2,x')|\leq C|F_{p_2p_2}(p_1,p_2,z_1,z_2,x')|, \\
& |F_{p_1p_2}(p_1+q_1,p_2+q_2,z_1,z_2,x')|\leq C|F_{p_1p_2}(p_1,p_2,z_1,z_2,x')|,
\end{split}\end{equation}
for any~$p_1,p_2,q_1,q_2\in\R^n$ with~$|q_1|\leq|p_1\cdot e_n|/4$ and~$|q_2|\leq|p_2\cdot e_n|/4$. We will make these assumptions throughout the paper.

Given~$R>0$, we consider the following energy functional
$$ \mathcal E_R(u,v):=\int_{\Omega\cap B_R} F(\nabla u(x),\nabla v(x),u(x),v(x),x') dx, $$
where, as usual, $B_R$ denotes the ball of radius~$R$ centered at the origin.

Following the line in~\cite{SaVa}, we want to study the symmetry properties
of minimal or stable solutions for the functional above
among suitable perturbations, obtained by domain deformations
in the direction given by~$e_n$.
To do this, we recall the following definitions introduced in~\cite{SaVa}.

\begin{defn}\label{deform}
Given a function~$w$, an~$e_n$-Lipschitz deformation of~$w$ in~$B_R$ is
a function~$\overline w$ defined by
$$ \overline w(x)=w(x+\varphi(x)e_n) \qquad \mbox{for any } x\in\Omega, $$
where~$\varphi$ is a Lipschitz function with compact support in~$B_R$,
with $\|\partial_n\varphi\|_{L^{\infty}(\R^n)}<1$.
\end{defn}

\begin{defn}\label{pert}
Given a function~$w\in C^{0,1}(\Omega)$, a piecewise~$e_n$-Lipschitz deformation
of~$w$ in~$B_R$ is a function~$\overline w\in C^{0,1}(\Omega)$ defined by
$$ \overline w(x) =\overline w^{(i)}(x) \qquad \mbox{for some $i$ (depending on $x\in\Omega$)}, $$
where~$\overline w^{(1)},\ldots,\overline w^{(m)}$ constitute a finite number
of~$e_n$-Lipschitz deformations of~$w$ in~$B_R$.
In this case, we write~$\overline w\in D_R(w)$.

Also, if all~$\overline w^{(i)}$ satisfy
$$ \overline w^{(i)}(x)=w(x+\varphi^{(i)}(x)e_n) \qquad \mbox{with } \|\varphi^{(i)}\|_{C^{0,1}(\Omega)}\leq\delta $$
for some~$\delta$, we write~$\overline w\in D^{\delta}_R(w)$.
\end{defn}

We recall also some elementary properties of a piecewise~$e_n$-Lipschitz deformation,
which follow easily from the definitions above:

\begin{prop}
The following properties hold true:
\begin{itemize}
\item[i)]if $\overline w^1,\overline w^2\in D^{\delta}_R(w)$, then $\min\left\lbrace \overline w^1,\overline w^2\right\rbrace, \max\left\lbrace \overline w^1,\overline w^2\right\rbrace\in D^{\delta}_R(w)$;
\item[ii)] if $\overline w\in D^{\delta}_R(w), \tilde w\in D^{\delta}_R(\overline w)$, then $\tilde w\in D^{3\delta}_R(w)$;
\item[iii)] if $\overline w\in D^{\delta}_R(w)$, then $\|\overline w-w\|_{L^{\infty}(\Omega)} \leq C\delta\|w\|_{C^{0,1}(\Omega)}$;
\item[iv)] if $\overline w\in D^{\delta}_R(w), w\in C^{1,1}(\Omega)$, then
$\|\overline w-w\|_{C^{0,1}(\Omega)} \leq C\delta\|w\|_{C^{1,1}(\Omega)}$.
\end{itemize}
\end{prop}

In the sequel we will also assume a growth condition on the functional~$\mathcal E$.
Namely, we suppose that there exists a constant $C>0$ such that,
for~$R$ sufficiently large,
\begin{equation}\begin{split}\label{growth}
&\int_{\Omega\cap B_R}|F_{p_1p_1}(\nabla u,\nabla v,u,v,x')||\nabla u|^2 +
|F_{p_2p_2}(\nabla u,\nabla v,u,v,x')||\nabla v|^2 \\
&\qquad\qquad\qquad\qquad\qquad + |F_{p_1p_2}(\nabla u,\nabla v,u,v,x')||\nabla u||\nabla v| dx\leq C R^2.
\end{split}\end{equation}

\subsection{A monotonicity result for minimizers of the energy functional}

The first result here deals with~$e_n$-minimizers of~$\mathcal E$.
To state it, we give the following:

\begin{defn}
We say that~$(u,v)$, with~$u,v\in C^{0,1}(\Omega)$, is an~$e_n$-minimizer for~$\mathcal E$
if for any~$R>0$ we have that~$\mathcal E_R(u,v)$ is finite and
$$ \mathcal E_R(u,v)\leq \mathcal E_R(\overline u,\overline v), \qquad \mbox{for any } \overline u\in D_R(u) \mbox{ and any } \overline v\in D_R(v). $$
\end{defn}

Now, we are in the position to state our first monotonicity result.
\begin{theorem}\label{T1}
Let~$u,v\in C^{1}(\Omega)$, and let~$F$ satisfy~\eqref{F1} and~\eqref{F2}.
Suppose that~$(u,v)$ is an~$e_n$-minimizer for the energy~$\mathcal E$,
and that the growth condition~\eqref{growth} is satisfied.

Then, $u$ and $v$ are monotone on each line in the~$e_n$-direction,
i.e., for any~$x\in\Omega$, either~$u_n(x+te_n)\geq0$ or~$u_n(x+te_n)\leq0$,
and either~$v_n(x+te_n)\geq0$ or~$v_n(x+te_n)\leq0$, for any~$t\in\R$.
\end{theorem}

\begin{rem}\label{import}
We notice that if a continuous function~$w$ is monotone on each line
in~$\R^n$, then it is one-dimensional, that is there exist a function~$w_0:\R\rightarrow\R$
and a unit direction~$\omega\in\S^{n-1}$ such that~$w(x)=w_0(\omega\cdot x)$
(see Section~$9$ in~\cite{SaVa} for a proof of this fact).
\end{rem}

\subsection{A monotonicity result for stable solutions}

The second result that we state concerns stable critical points
of the energy instead of~$e_n$-minimizers.
We give a definition of the stability condition that involves
the second variation of the functional~$\mathcal E$ for deformations
of the solution both in the~$e_n$-direction and in the vertical~$e_{n+1}$-direction.
Precisely, we give the following:

\begin{defn}
Given a function~$w$, a piecewise Lipschitz deformation of~$w$ in
the~$\left\lbrace e_n,e_{n+1}\right\rbrace$-directions is a function~$\tilde w$
defined as
$$ \tilde w=\overline w +\phi, $$
where~$\overline w\in D^{\delta}_R(w)$ and~$\phi$ is a Lipschitz function
with compact support in~$\Omega\cap B_R$, and~$\|\phi\|_{C^{0,1}(\Omega)}\leq\delta$.
In this case we write~$w\in\mathcal D^{\delta}_R(w)$.
\end{defn}

\begin{defn}\label{stab}
We say that~$(u,v)$ is an~$\left\lbrace e_n,e_{n+1}\right\rbrace$-stable
solution of~$\mathcal E$ if for every~$R>0$ and~$\epsilon>0$ there exists~$\delta>0$ depending
on~$R$, $\epsilon$, $u$ and $v$ such that, for any~$t\in(0,\delta)$, we have
that~$\mathcal E_R(u,v)$ is finite and
\begin{equation}\label{stabineq}
\mathcal E_R(\tilde u,\tilde v)-\mathcal E_R(u,v)\geq -\epsilon t^2, \qquad \mbox{for any } \tilde u\in\mathcal D^t_R(u) \mbox{ and any } \tilde v\in\mathcal D^t_R(v).
\end{equation}
\end{defn}

\begin{rem}\label{senss}
As we shall see later on (see Lemma \ref{sens}), this notion of $\{e_n,e_{n+1}\}$-stable solution shares intimate links with the classical notion of stability.
This actually follows from the same arguments as in the case of a single equation (see \cite{SaVa}).
Actually, we can infer directly from the definition that
\begin{enumerate}
\item $(u,v)$ is a classical minimizer of $\mathcal E$ $\Rightarrow$ $(u,v)$ is $\{e_n,e_{n+1}\}$-stable for $\mathcal E$;
\item $(u,v)$ is $\{e_n,e_{n+1}\}$-stable for $\mathcal E$ $\Rightarrow$ $(u,v)$ is a critical point for $\mathcal E$.
\end{enumerate}
This is because we allow perturbation in the $e_{n+1}$-direction.
\end{rem}

Next, we state our monotonicity result for~$\left\lbrace e_n,e_{n+1}\right\rbrace$-stable solutions.

\begin{theorem}\label{T2}
Let~$F\in C^{3,\alpha}(\R^{2n}\times\R^2\times V)$ satisfy~\eqref{F2} and $(u,v)$ be such that
\begin{align}
&\text{either}~~  u,v\in C^{0,1}(\Omega)~~\text{are convex} \nonumber \\
&\text{or}~~  u,v\in C^{1,1}(\Omega). \nonumber
\end{align}
Moreover, suppose that~$(u,v)$ is an~$\left\lbrace e_n,e_{n+1}\right\rbrace$-stable
solution of~$\mathcal E$, and that the growth condition~\eqref{growth} holds
true.

Then, $u$ and~$v$ are monotone in the~$e_n$-direction,
i.e. either~$u_n\geq 0$ or $u_n\leq 0$ and either~$v_n\geq0$ or $v_n\leq0$ in~$\Omega$.
\end{theorem}

By Theorem~\ref{T2} and Remark~\ref{import} one obtains that
an~$\left\lbrace e_n,e_{n+1}\right\rbrace$-stable solution of~$\mathcal E$
is one-dimensional, in the sense that
there exist $\overline u$, $\overline v:\R\to\R$ and $\omega_u$, $\omega_v\in\S^{n-1}$ such that
$(u(x),v(x))=\big(\overline u(\omega_u\cdot x), \overline v(\omega_v\cdot x)\big)$.
We remark that, in general, it is not possible to conclude that~$u$ and~$v$ have the same direction
of monotonicity. It is not true, for instance, for an uncoupled system of PDEs, such as
\begin{eqnarray*}
\left\{
\begin{array}{r c l}
\Delta u=0,\\
\Delta v=0.
\end{array}
\right.
\end{eqnarray*}
See also Remark~1.4 in~\cite{Di} for a discussion on this fact.

On the other hand, there are cases in which it is possible to obtain
that~$u$ and~$v$ have the same direction of monotonicity, even for quite general systems,
see e.g.~\cite{Di} and~\cite{pinamonti}.

We stress that the aim of this paper is to generalize
to systems a method to obtain monotonicity properties
that adapts to quite general settings.

\subsection{Organization of the paper} In Section~\ref{sec:stab}
we show that, under suitable assumptions, the notion of
$\left\lbrace e_n,e_{n+1}\right\rbrace$-stability is equivalent to the
notion of classical stability.
In Section~\ref{sec:local} we perform a local analysis and, in Section~\ref{sec:infi},
we estimate the energy of the perturbation at infinity,
collecting the basic facts in order to prove Theorems~\ref{T1} and~\ref{T2} in
Sections~\ref{sec:proof} and~\ref{sec:proof2}, respectively.
Finally, in Section~\ref{applications} we investigate some applications
of our results in some concrete cases.

\section{Stability property}\label{sec:stab}

The following lemma justifies Remark \ref{senss}.

\begin{lemma}\label{sens}
If $\Omega=\mathbb{R}^n$, $F\in C^2$ and $u$, $v\in C^2(\mathbb{R}^n)$,
then $(u,v)$ is a classical stable solution for the energy functional $\mathcal E$
if, and only if, $(u,v)$ is an $\left\lbrace e_n,e_{n+1}\right\rbrace$-stable
solution of~$\mathcal E$ according to Definition~\ref{stab}.
\end{lemma}

\begin{proof}
Suppose first that $(u,v)$, with $u$, $v\in C^2(\mathbb{R}^n)$, is a classical stable solution.
That is, by definition,
\begin{equation}
\liminf_{t\to0}\frac{\mathcal E_R(u+t\varphi^1,v+t\varphi^2)-\mathcal E_R(u,v)}{t^2}\geq 0
\quad {\mbox{ for any }} \varphi^1,\varphi^2\in C^{0,1}_0(B_R). \label{stabbb}
\end{equation}
In addition, setting
\begin{equation}\label{9910-1}
\Delta_{\varphi^1,\varphi^2}^1 F:=F(\nabla u+t\nabla\varphi^1,\nabla v+
t\nabla\varphi^2,u+t\varphi^1,v+t\varphi^2,x')-F(\nabla u,\nabla v,u,v,x'),
\end{equation}
we have
\begin{equation}\label{9910}
\Delta_{\varphi^1,\varphi^2}^1
F=tH(\varphi^1,\varphi^2)+\frac{t^2}{2}L(\varphi^1,\varphi^2)+o(t^2),
\end{equation}
where
\begin{eqnarray*}
&& H(\varphi^1,\varphi^2):= F_{(p_1)_i}\varphi_i^1+F_{z_1}\varphi^1+F_{(p_2)_i}\varphi_i^2
+F_{z_2}\varphi^2 \\
{\mbox{and }} && L(\varphi^1,\varphi^2):=
F_{(p_1)_i(p_1)_j}\varphi_i^1\varphi_j^1+F_{z_1z_1}(\varphi^1)^2+
2F_{(p_1)_i z_1}\varphi^1\varphi_i^1 \\
&&\qquad\qquad\qquad +F_{(p_2)_i(p_2)_j}\varphi_i^2\varphi_j^2
+F_{z_2z_2}(\varphi^2)^2+2F_{(p_2)_i z_2}\varphi^2\varphi_i^2  \\
&&\qquad\qquad\qquad +F_{(p_1)_i(p_2)_j}\varphi_i^1\varphi_j^2+2F_{z_2z_1}
\varphi^1\varphi^2+2F_{(p_2)_i z_1}\varphi^2\varphi_i^2 \\
&&\qquad\qquad\qquad +F_{(p_2)_i(p_1)_j}\varphi_i^2\varphi_j^1+
2F_{(p_1)_i z_2}\varphi^2\varphi_i^1
\end{eqnarray*}
and the derivatives of $F$ are evaluated at $(\nabla u,\nabla v,u,v,x')$.

Now we observe that, since $(u,v)$ is a critical point of $\mathcal{E}$,
$$ \int_{B_R} H(\varphi^1,\varphi^2) =0.$$
So integrating \eqref{9910} and recalling \eqref{9910-1}, we obtain that
\begin{equation}
\mathcal E_R(u+t\varphi^1,v+t\varphi^2)-\mathcal E_R(u,v)
=\frac{t^2}{2}\int_{B_R}L(\varphi^1,\varphi^2)\mathrm{d}x+o(t^2). \label{eggg}
\end{equation}
Dividing by $t^2$ and recalling \eqref{stabbb} we find that
\[\int_{B_R}L(\varphi^1,\varphi^2)\mathrm{d}x\geq 0. \]
Thus, \eqref{eggg} yields
\begin{equation}\label{234ert}
\mathcal E_R(u+t\varphi^1,v+t\varphi^2)-\mathcal E_R(u,v)\geq o(t^2). \end{equation}
Given $\tilde{u}\in\mathcal D_R^t(u)$ and $\tilde{v}\in \mathcal D_R^t(v)$
we may choose $\varphi^1:=\frac{\tilde{u}-u}{t}$ and $\varphi^2:=\frac{\tilde{v}-v}{t}$.
By construction, the Lipschitz norm of $\varphi^1$ and $\varphi^2$
is bounded by a quantity that does not depend on $t$.
Therefore \eqref{234ert} implies \eqref{stabineq},
and so it follows that $(u,v)$ is $\{e_n,e_{n+1}\}$-stable.

Reciprocally, let $(u,v)$ be an $\{e_n,e_{n+1}\}$-stable solution
of $\mathcal E$ (and so, a critical point).
In this case, we choose $\tilde{u}:=u+t\varphi^1$ and $\tilde{v}:=v+t\varphi^2$
in \eqref{stabineq}. Then, taking $\epsilon$ arbitrarily small, we prove \eqref{stabbb}.
Therefore $(u,v)$ is a stable solution of $\mathcal{E}$.
This completes the proof.
\end{proof}

\section{Local analysis}\label{sec:local}

In this section we will perform the local perturbation of a couple $(u,v)$.
More precisely, we will show that we can perturb the couples
$$ \left(\max\{u(x), u(x)+t e_n\}, v(x)\right) \quad {\mbox{ and }}
\quad \left(u(x), \max\{v(x), v(x)+t e_n\}\right) $$
in such a way that the energy of the ``perturbed couples'' decreases.

We first show that if one of the two elements of a couple $(u,v)$
is a maximum of two functions that form an angle at an intersection point,
then it cannot be an $e_n$-minimizer for $\mathcal{E}$.

\begin{lemma}\label{lem1}
Suppose that~$0\in\Omega$, and that~$u,u^1,v,v^1$ are~$C^1$ functions which satisfy
\begin{equation}\label{ip900}
u(0)=u^1(0), \qquad u^1_n(0)<0<u_n(0),
\end{equation}
and
$$ v(0)=v^1(0), \qquad v_n(0)<0<v^1_n(0). $$

Then, the couples~$g_1:=\left(\max\left\lbrace u,u^1\right\rbrace,v\right)$
and~$g_2:=\left(u,\max\left\lbrace v,v^1\right\rbrace\right)$ are
not~$e_n$-minimizers for~$\mathcal E$ in any ball~$B_\eta$ with~$\eta>0$.
\end{lemma}

\begin{proof}
We prove the lemma for the couple~$g_1$, the proof for~$g_2$ being
similar with obvious modifications.

We argue by contradiction and we assume that~$g_1$ is an~$e_n$-minimizer in a ball~$B_\eta$,
for some~$\eta>0$.
We define~$F_0(p_1,p_2):=F(p_1,p_2,0,0,0)$.
Moreover, we observe that if
we subtract a linear functional from~$F$ then this does not affect the minimality.
Indeed, we can consider
$$ \tilde F(p_1,p_2,z_1,z_2,x):=F(p_1,p_2,z_1,z_2,x) - p_0\cdot p_1, $$
then, if~$\tilde{\mathcal E}$ is the associated energy functional,
we have
$$  \mathcal E_R(u,v)-\tilde{\mathcal E}_R(u,v)=\int_{B_R}p_0\cdot\nabla u\, dx =\int_{\partial B_R}p_0 u\cdot\nu\, d\sigma. $$
This means that the difference between~$\mathcal E_R(u,v)$ and~$\tilde{\mathcal E}_R(u,v)$
is a term which depends only on the boundary value of~$u$.
Therefore, if~$(u,v)$ is an~$e_n$-minimizer for~$\mathcal E$, then it is
an~$e_n$-minimizer for~$\tilde{\mathcal E}$ as well.
From this we deduce that we may assume
\begin{equation}\label{wqerr4343}
F_0(\nabla u(0),\nabla v(0))=F_0(\nabla u^1(0),\nabla v(0)). \end{equation}

Moreover, we can suppose that~$u(0)=u^1(0)=0$; for this,
it is sufficient to translate~$F$ in the variable~$z_1$.

Now, for small~$r>0$, we define the rescaled functions
$$ u_r(x):=r^{-1}u(rx), \qquad u^1_r(x):=r^{-1}u^1(rx) \quad {\mbox{ and }} \quad
v_r(x):=r^{-1}v(rx). $$
We also set~$g_r:=\left(\max\left\lbrace u_r,u^1_r\right\rbrace, v_r\right)$.
If
$$ F_r(p_1,p_2,z_1,z_2,x):= F(p_1,p_2,rz_1,rz_2,rx) $$
is the rescaled functional, we have that~$g_r$ is an~$e_n$-minimizer
for~$F_r$ in~$B_{\eta/r}$.
Sending~$r\rightarrow0^+$, we obtain the following limits, which are uniform on compact sets:
\begin{equation}\begin{split}\label{unif}
& F_r \rightarrow F_0(p_1,p_2), \\
& u_r(x) \rightarrow u_0(x):= \nabla u(0)\cdot x, \qquad \nabla u_r \rightarrow \nabla u_0, \\
& u^1_r(x)\rightarrow u^1_0(x):=\nabla u^1(0)\cdot x, \qquad \nabla u^1_r\rightarrow\nabla u^1_0, \\
& v_r(x)\rightarrow v_0(x):=\nabla v(0)\cdot x, \qquad \nabla v_r\rightarrow\nabla v_0.
\end{split}\end{equation}

Now, we set~$f_0:=\max\left\lbrace u_0,u^1_0\right\rbrace$ and we consider
the couple~$g_0=(f_0,v_0)$.
The fact that~$F$ is strictly convex in the variable~$p_1$ implies that
\begin{equation}\label{agg123}
{\mbox{$g_0$ is not a minimizer for~$F_0$.}}\end{equation}
To see this,
given the function
$$\beta_R(x'):=c \delta \max\left\lbrace 0,|x'|-R\right\rbrace, $$
with~$R$ large and $c, \delta$ small, we define
\begin{equation}\label{conv99}
h_0:= 1+\alpha u_0 + (1-\alpha) u^1_0 - \beta_R(x'),
\end{equation}
where~$\alpha\in(0,1)$ is small.
Then, we have that
\begin{equation}\label{ufff45}
{\mbox{the couple $(\max\left\lbrace f_0,h_0\right\rbrace, v_0)$
agrees with~$g_0$ outside the ball~$B_{R+R_1}$.}}\end{equation}
On the other hand, in~$B_R$, considering~$h_0$, we cut the graphs of two
transversal linear functions by a single one.
Therefore, if we take~$R$ sufficiently large,
\begin{equation}\label{24356ty}
{\mbox{the couple $(\max\left\lbrace f_0,h_0\right\rbrace, v_0)$ has lower energy
for~$F_0$ than~$g_0$.}}\end{equation}
To see this, we observe that, by definition, $\nabla f_0$ coincides with either $\nabla u_0$
or $\nabla u_0^1$, and so from \eqref{wqerr4343} we have that
$$ F_0(\nabla u_0,\nabla v_0)=F_0(\nabla u^1_0,\nabla v_0) = F_0(\nabla f_0, \nabla v_0).$$
Therefore, by the strict convexity of $F_0$ with respect to the first variable, we have
that, fixed $\alpha$ and $\delta$,
\begin{equation}\begin{split}\label{qw432}
F_0(\nabla h_0, \nabla v_0) =\,& F_0\big( (\alpha\nabla u_0 +(1-\alpha)\nabla u_0^1, \nabla v_0\big)\\
\le\, & \alpha F_0(\nabla u_0,\nabla v_0) +(1-\alpha)F_0(\nabla u_0^1,\nabla v_0) -\eta \\
=\,& F_0(\nabla f_0,\nabla v_0)-\eta
\end{split}\end{equation}
in $B_R\cap \{h_0>f_0\}$, for some $\eta>0$. On the other hand,
the set $\{h_0>f_0\}$ is contained in a strip in the $e_n$-direction.
This and \eqref{ufff45} imply that, when we integrate \eqref{qw432} in $B_R$,
then the energy of $(\max\left\lbrace f_0,h_0\right\rbrace, v_0)$ is less than the energy of $g_0$
minus a term of order $\underline{C}\eta R^{n-1}$,
whereas, when we integrate in $B_{R+R_1}$,
the energy of $(\max\left\lbrace f_0,h_0\right\rbrace, v_0)$ is less than the energy of $g_0$
minus a term of order $\overline{C}\eta R^{n-2}$.
All in all, if $R$ is large enough we obtain \eqref{24356ty}.
In turn, this implies~\eqref{agg123}.

Now we set
\begin{equation}\label{fr}
f_r:=\max\left\lbrace u_r,u^1_r\right\rbrace.
\end{equation}
Since the convergence in~\eqref{unif} is uniform,
we have that
$$ h_r:= (\max\left\lbrace f_r, h_0\right\rbrace, v_r) $$
has lower energy for~$F_r$ than~$g_r$, as well.

Hence, we can scale back and obtain that~$h_*(x):=r h_r(x/r)$
has lower energy for~$F$ in~$B_{r(R+R_1)}\subseteq B_{\eta}$
than the one of~$g_1$.

In order to get a contradiction, it remains to prove
that~$h_*$ is an allowed perturbation of~$g_1$ (according to Definition \ref{pert}).
Notice that this is equivalent to check that~$h_r$ is an allowed perturbation of~$g_r$.
That is, recalling \eqref{fr}, we have to prove that
\begin{equation}\begin{split}\label{allowed}
&{\mbox{$\max\left\lbrace f_r,h_0\right\rbrace$ is
a piecewise Lipschitz domain deformation of~$f_r$}}\\
&{\mbox{with the Lipschitz norm bounded by~$\delta$.}}
\end{split}\end{equation}
To do this,
we recall the uniform convergence of~$u_r$ and~$u^1_r$ to~$u_0$ and~$u^1_0$, respectively,
given by \eqref{unif}, and the definition of $h_0$ given in \eqref{conv99} to obtain that
\begin{equation}\label{efrettg76}
h_0(x)= 1+\alpha u_r(x) + (1-\alpha) u^1_r(x) - \beta_R(x') +\omega_r(x),
\end{equation}
where $\omega_r\to0$ as $r\to0^+$ locally uniformly, together with its derivatives.

Now we notice that our hypothesis in \eqref{ip900} gives
that~$\nabla u^1_0\cdot e_n<0<\nabla u_0\cdot e_n$.
This, together with the uniform convergence in \eqref{unif},
implies that we can apply the Implicit Function Theorem, and we have that
the part of the graph of~$\max\left\lbrace f_r,h_0\right\rbrace$ where~$h_0>f_r$
can be obtained from~$u_r$ by a Lipschitz domain deformation with Lipschitz norm
less than~$\delta$, provided that we take~$\alpha$ sufficiently small.
Indeed, fixed $x'$ (and so looking at the 1-dimensional problem in the last variable only)
and recalling \eqref{efrettg76}, we obtain that
$$ u_r(x', x_n + \varphi(x)) = h_0(x) =
1+\alpha u_r(x) + (1-\alpha) u^1_r(x) - \beta_R(x') +\omega_r(x),$$
thanks to the Implicit Function Theorem in 1-dimension.
Furthermore, if $\alpha$ and $r$ are sufficiently small, the perturbation function $\varphi$
has norm bounded by $\delta/2$. This shows \eqref{allowed}
and finishes the proof of Lemma \ref{lem1}.
\end{proof}

Now we deal with perturbations of the couples
$$ \left(\max\{u(x), u(x)+t e_n\}, v(x)\right) \quad {\mbox{ and }}
\quad \left(u(x), \max\{v(x), v(x)+t e_n\}\right) $$
with lower energy.

\begin{lemma}\label{lem2}
Let~$u,v\in C^2(\Omega)$ be such that $(u,v)$ is a critical point for the functional~$\mathcal E$ in
a neighborhood of the origin, and let~$F\in C^2$ in a neighborhood
of $(\nabla u(0),\nabla v(0),u(0),v(0),0)$.
Suppose that
$$ u_n(0)=0, \qquad \nabla u_n(0)\neq 0, $$
and set
\begin{equation}\label{w1}
w^1(x):= \max\left\lbrace  u(x), u(x+te_n)\right\rbrace.
\end{equation}

Then, for every~$\eta>0$ there exists a function~$\psi^1$ which is Lipschitz and has
compact support in~$B_{\eta}$ such that, for any small~$t$,
$$ \mathcal E_{\eta}(w^1 +t\psi^1,v)-\mathcal E_{\eta}(w^1,v) \leq -c\, t^2, $$
where~$c>0$ is a small constant that depends on~$F$, $\eta$ and~$u$.
\end{lemma}

\begin{rem}\label{remlem}
We point out that we have an analogous result if we consider the function~$v$ instead of~$u$.
Precisely, if we assume that
$$ v_n(0)=0, \qquad \nabla v_n(0)\neq 0, $$
and we consider
$$ w^2(x):=\max\left\lbrace v(x),v(x+te_n)\right\rbrace, $$
then, for every~$\eta>0$ there exists a function~$\psi^2$ which is Lipschitz and has
compact support in~$B_{\eta}$ such that, for any small~$t$,
$$ \mathcal E_{\eta}(u,w^2 +t\psi^2)-\mathcal E_{\eta}(u,w^2) \leq -c\, t^2, $$
where~$c>0$ is a small constant depending on~$F,\eta$ and~$v$.
\end{rem}

\begin{proof}[Proof of Lemma \ref{lem2}]
We define
\begin{equation}\label{u1}
u^1(x):= \frac{u(x+te_n)-u(x)}{t}
\end{equation}
and we observe that
\begin{equation}\label{kuy6}
\|u^1-u_n\|_{C^{0,1}(B_{\eta})}=o(1) \quad \mbox{as } t\rightarrow 0. \end{equation}

We consider a Lipschitz function~$g_1$ and, using the fact that~$F\in C^2$ in
the variables~$p_1$ and~$z_1$, we compute
\begin{eqnarray*}
&&F(\nabla u+t\nabla g_1,\nabla v,u+tg_1,v,x') \\
&=& F(\nabla u,\nabla v,u,v,x') + t\left(F_{p_1}\nabla g_1 +F_{z_1}g_1\right) \\
&&\quad + t^2\left(\left(\nabla g_1\right)^{T}F_{p_1p_1} \nabla g_1 + F_{z_1z_1} g_1^2 + 2g_1 F_{p_1z_1}\cdot\nabla g_1\right) +o(t^2),
\end{eqnarray*}
where the derivatives of~$F$ are evaluated at~$(\nabla u,\nabla v,u,v,x')$
and the constant in the error term~$o(t^2)$ depends on~$u$, $F$ and~$\|g_1\|_{C^{0,1}(B_{\eta})}$.
Hence, we obtain
$$ \mathcal E_{\eta}(u+tg_1,v) =\mathcal E_{\eta}(u,v) +t\, L(g_1) +t^2 Q(g_1) +o(t^2), $$
with
$$ L(g_1):=\int_{B_{\eta}}\left(F_{p_1}\cdot\nabla g_1 + F_{z_1}g_1\right) dx $$
and
\begin{eqnarray*}
Q(g_1) &:=& \int_{B_{\eta}} G(\nabla g_1,g_1,x)dx \\
&=& \int_{B_{\eta}}\left(\left(\nabla g_1\right)^{T}F_{p_1p_1}\nabla g_1 + F_{z_1z_1}g^2_1
+2g_1 F_{p_1z_1}\cdot\nabla g_1\right)dx.
\end{eqnarray*}

Now, we take a Lipschitz function~$\psi^1$ with compact support in~$B_{\eta}$,
and we use the fact that~$(u,v)$ is a critical point for~$\mathcal E$ to obtain
\begin{equation}\label{Eeta}
\mathcal E_{\eta}(u+tu^1_+ +t\psi^1,v) -\mathcal E_{\eta}(u+tu^1_+,v) = t^2\left(Q(u^1_+ +\psi^1) -Q(u^1_+)\right) + o(t^2).
\end{equation}
Using~\eqref{w1} and~\eqref{u1}, we can write
\begin{equation}\label{111}
w^1=u+tu^1_+.
\end{equation}
Now we claim that, for~$\eta$ sufficiently small,
\begin{equation}\label{qwertp090}
Q(u^1_+)-Q((u_n)_+)=o(1) \quad \mbox{ and }\quad Q(u^1_+ +\psi^1)-Q((u_n)_+ +\psi^1)=o(1),
\quad {\mbox{ as }} t\to0.
\end{equation}
We focus on the first equality in \eqref{qwertp090}, since the second is similar.
To prove it, fixed $\mu>0$, we define
$$ \mathcal{B}_\mu^1:=B_\eta \cap \{|u_n|\le\mu\} \quad {\mbox{ and }} \quad
\mathcal{B}_\mu^2:=B_\eta \cap \{|u_n|>\mu\}.  $$
Notice that \eqref{kuy6} implies that
\begin{equation}\label{kuy6-2}
\lim_{t\to0} \|u^1_+ -(u_n)_+\|_{C^{0,1}(\mathcal{B}_\mu^2)} =
\lim_{t\to0} \|u^1-u_n\|_{C^{0,1}(\mathcal{B}_\mu^2)}=0.\end{equation}
As for the contribution coming from $\mathcal{B}_\mu^1$, we observe that,
since $\nabla u_n(0)\neq0$, the measure of $\mathcal{B}_\mu^1$ is at most of order of $\mu$.
This, together with \eqref{kuy6-2}, gives that
$$ \lim_{t\to0}| Q(u^1_+)-Q((u_n)_+)| \le C\mu, $$
for some $C>0$. Since $\mu$ can be taken arbitrarily small, this implies \eqref{qwertp090}.

Formula \eqref{qwertp090} means that, for~$\eta$ sufficiently small, we can write~$(u_n)_+$ instead of~$u^1_+$
in the right hand side of~\eqref{Eeta}.
Hence, recalling also~\eqref{111}, we have that
\begin{equation}\label{222}
\mathcal E_{\eta}(w^1+t\psi^1,v)-\mathcal E_{\eta}(w^1,v)= t^2\left(Q((u_n)_+ +\psi^1) -Q((u_n)_+)\right) + o(t^2).
\end{equation}
Now, we notice that~$u_n$, $0$ and~$G$ satisfy the hypotheses of Remark $4.3$ in~\cite{SaVa},
and therefore~$(u_n)_+$ is not a minimizer of~$Q$.
So we can take the function~$\psi^1$ in such a way that
$$ Q((u_n)_+ +\psi^1)\leq Q((u_n)_+)-c $$
for some small~$c$, which may depend on~$u$, $F$ and~$\eta$.
This together with~\eqref{222} gives that, for small~$t$,
$$ \mathcal E_{\eta}(w^1+t\psi^1,v)-\mathcal E_{\eta}(w^1,v)\leq -c t^2 $$
and this concludes the proof.
\end{proof}

Now we show that, under an additional regularity hypothesis on $F$,
the non-degeneracy condition $\nabla u_n\neq 0$ in Lemma \ref{lem2} is satisfied.
For this we use the Hopf Lemma, by adapting the proof of Lemma 4.6 in \cite{SaVa}
to the slightly more delicate case of a system of equations.

\begin{lemma}\label{beH}
Assume that
\begin{align}
&\text{either}~~  u,v\in C^{0,1}(\Omega)~~\text{are convex} \nonumber \\
&\text{or}~~  u,v\in C^{1,1}(\Omega). \nonumber
\end{align}
Furthermore, assume that $(u,v)$ is a critical point for
$\mathcal{E}$ and $F\in C^{3,\alpha}$ in a neighborhood of $(\nabla u(0),\nabla v(0), u(0),v(0),0)$.

Then, $u$ and $v$ are of class $C^{3,\alpha}$ in a neighborhood of $0$.

If, in addition, $u_n(0)=0$ (resp. $v_n(0)=0$) and $u_n$ (resp. $v_n$) does not vanish identically
in a neighborhood $V_0$ of $0$, then there exists a point $x_0\in V_0$ (resp. $x_1\in V_0$) such that
\begin{align}
u_n(x_0)=0,~~\nabla u_n(x_0)\ne0~~~~\text{and}~~~~v_n(x_1)=0,~~\nabla v_n(x_1)\ne0 \nonumber
\end{align}
\end{lemma}

\begin{proof}
Since $(u,v)$ is a critical point for $\mathcal{E}$ it satisfies the elliptic system of equation
\[
\left\{
\begin{array}{c c l}
G^1(\nabla^2u,\nabla^2v,\nabla u,\nabla v, u,v,x'):=\mathrm{div}_x F_{p_1}(M)-F_{z_1}(M) & = & 0 \\
G^2(\nabla^2u,\nabla^2v,\nabla u,\nabla v, u,v,x'):=\mathrm{div}_x F_{p_2}(M)-F_{z_2}(M) & = & 0
\end{array}
\right.
\]
where $M:=(\nabla u,\nabla v,u,v,x')$.
Consider the first equation of the above system, with $v$ fixed.
If $v\in C^{0,1}(\Omega)$ is convex, then $\nabla v\in L^\infty(\Omega)^n$ and
$\nabla^2 v\in L^\infty(\Omega)^{n\times n}$ exist
almost everywhere (this follows from Rademacher Theorem and Alexandrov Theorem).
On the other hand, if $v\in C^{1,1}(\Omega)$, then $\nabla^2 v\in L^\infty(\Omega)^{n\times n}$
also exists almost everywhere (this follows from Rademacher Theorem).
Thus, at fixed $v$, the first equation (in the variable $u=u_v$)
is satisfied in the classical sense and the corresponding
solution $u_v$ belongs to $C^{3,\alpha}(\Omega)$
(this follows from Theorem 2.1 in \cite{Trud}, the Schauder estimates
and the fact that $F\in C^{3,\alpha}$).
In like manner, for any fixed $u$, the solution $v_u$
of the second\footnote{We may actually consider indifferently
the first or the second equation, as only one equation is needed.}
equation belongs to $C^{3,\alpha}(\Omega)$.
Therefore, $(u,v)\in C^{3,\alpha}(\Omega)\times C^{3,\alpha}(\Omega)$.
This shows the first claim of Lemma~\ref{beH}.
\medskip

Now we focus on the second claim of Lemma~\ref{beH}. To prove it, we observe
that
\begin{equation}\label{pqw987}
G^1=G^1(q_1,q_2,p_1,p_2,z_1,z_2,x')\in C^{1,\alpha}.\end{equation}
Hence, differentiating in the $e_n$-direction we see that $w:=u_n\in C^{2,\alpha}(\Omega)$ satisfies the linearized equation (in the viscosity sense)
\begin{equation}
G_{(q_1)_{ij}}^1w_{ij}+G_{(p_1)_i}^1w_i+G_{z_1}^1w
=-G_{(q_2)_{ij}}^1(v_n)_{ij}-G_{(p_2)_i}^1(v_n)_{i}-G_{z_2}^1v_n. \label{linEq}
\end{equation}
Here, the derivatives of $G^1$ are evaluated at
$(\nabla^2u,\nabla^2v,\nabla u,\nabla v, u,v,x')$.
For the convenience of the reader, we rewrite \eqref{linEq} as
$$ L^1w=f_1, $$
where
\begin{eqnarray*}
&& L^1 w := G_{(q_1)_{ij}}^1w_{ij}+G_{(p_1)_i}^1w_i+G_{z_1}^1w \\
{\mbox{and }} && f_1 :=-G_{(q_2)_{ij}}^1(v_n)_{ij}-G_{(p_2)_i}^1(v_n)_{i}-G_{z_2}^1v_n.
\end{eqnarray*}
Moreover, we set
\begin{eqnarray*}
P&:=&\{x\in V_0: f_1(x)> 0\}, \\
E&:=&\{x\in V_0:f_1(x)= 0\}\\
{\mbox{and }}\; N&:=&\{x\in V_0: f_1(x)< 0\}.
\end{eqnarray*}
In addition, denote by $(P_j)_{j\in J_1}$ (resp. $(E_j)_{j\in J_2}$)
the connected components of $P$ (resp. $E$).
This makes sense since $f_1\in C^{0,\alpha}(V_0)$
(recall \eqref{pqw987} and the fact that $v$ is $C^{3,\alpha}$ in $V_0$).

Note that if $P=N=\varnothing$ then $f_1\equiv 0$, and so
the conclusion follows as in Lemma 4.6 in \cite{SaVa}.
Thus we may suppose now that either $P$ or $N$ are nonempty.
Moreover, we can assume that $P$ is nonempty,
otherwise it is enough to replace $w$ by $-w$.

Also, we observe that
\begin{equation}\begin{split}\label{afster34}
&{\mbox{if there exists a compact set $\Sigma \Subset V_0$ such that $w_{|\Sigma}\equiv0$,}}\\
&{\mbox{then necessarily $\Sigma\subset E$.}}
\end{split}\end{equation}
Now we take\footnote{We use here the short notation $\{w=0\}=\{x\in V_0:w(x)=0\}$,
and similarly for $\{w<0\}$ and $\{w>0\}$.}
$x_0\in\{w=0\}\cap P$. Then there exists $j_1\in J_1$ such that $x_0\in P_{j_1}$.
It follows from \eqref{afster34} and the continuity of $f_1$
that $w$ cannot vanish identically in $P_{j_1}$.
Thus, we can apply the Hopf Lemma to $w$
at the point $x_0$, which admits a tangent ball
from either $\{w<0\}\cap P_{j_1}$ or $\{w>0\}\cap P_{j_2}$,
and the conclusion follows.

If $x_0\in\{w=0\}\cap N$,we replace $w$ with $-w$ and we reason as above.

If $x_0\in\{w=0\}\cap E$, we let $E_{j_2}$, for some $j_2\in J_2$,
be the connected component of $E$ such that $x_0\in E_{j_2}$.
By the continuity of $f_1$, there exists $j_1\in J_1$ such that
$K:=E_{j_2}\cup P_{j_1}$ is connected (up to exchanging $w$ with $-w$).
Again, \eqref{afster34} implies that $w$ cannot vanish identically in $K$, and so
we can apply the Hopf Lemma to $w$ at $x_0$,
which admits a tangent ball from either $\{w<0\}\cap K$ or $\{w>0\}\cap K$.

The same arguments apply to $v_n$. This completes the proof of Lemma \ref{beH}.
\end{proof}

\begin{rem}
The above result states that the non-degeneracy hypothesis $\nabla u_n\ne0$ of Lemma~\ref{lem2}
is always satisfied under a slightly stronger hypothesis on the potential $F$.
On the other hand, notice that we obtain a $C^{3,\alpha}$ regularity for $u$ and $v$ near 0, even though we asked only for, say, a $C^{1,1}$ regularity.
Thus, Lemma~\ref{lem2} is consistent.
\end{rem}

\section{Perturbations at infinity}\label{sec:infi}

In this section we modify the couples
$$ \left(\max\{u(x), u(x)+t e_n\}, v(x)\right) \quad {\mbox{ and }}
\quad \left(u(x), \max\{v(x), v(x)+t e_n\}\right) $$
at infinity in such a way that they become compact perturbations
of the couple $(u,v)$.

For this, for any~$R>1$ we define the function~$\varphi_R:\R\rightarrow\R$,
which is Lipschitz, even and with compact support, as
\begin{equation}\label{fiR}
\varphi_R(s):=\left\{
\begin{matrix}
1 & \mbox{if } 0\leq s\leq\sqrt{R}, \\
2\, \frac{\log R-\log s}{\log R}
& \mbox{if } \sqrt{R}<s<R, \\
0 & \mbox{if } s>R.
\end{matrix}
\right.
\end{equation}
We have that
\begin{equation}\label{fiRderiv}
\varphi'_R(s):=\left\{
\begin{matrix}
0 & \mbox{if } s\in(0,\sqrt{R})\cup(R,+\infty), \\
\frac{-2}{s\log R}
& \mbox{if } s\in(\sqrt{R},R). \\
\end{matrix}
\right.
\end{equation}
Also, for any~$t\in(0,\sqrt{R}/4]$, we consider the following bi-Lipschitz change of coordinates:
\begin{equation}\label{bilip}
x\mapsto y(x):= x+t\varphi_R(|x|)e_n.
\end{equation}
In these new coordinates, we define the functions~$u^+_{R,t}$ and~$v^+_{R,t}$ as
$$ u^+_{R,t}(y):=u(x) \quad {\mbox{ and }} \quad v^+_{R,t}(y)=v(x). $$
We observe that~$u^+_{R,t}(x)$ and~$v^+_{R,t}$ coincide with~$u(x-te_n)$
and~$v(x-te_n)$ respectively in~$B_{\sqrt{R}/2}$
and with~$u(x)$ and~$v(x)$ respectively outside~$B_R$.

We also define~$u^-_{R,t}$ and~$v^-_{R,t}$ by replacing~$t$ with~$-t$ in \eqref{bilip}.

Now, we want to obtain an estimate of~$\mathcal E_R(u^+_{R,t},v^+_{R,t})$
and~$\mathcal E_R(u^-_{R,t},v^-_{R,t})$
in terms of~$\mathcal E_R(u,v)$.
For this, we notice that
$$ D_x y=I+A, $$
where
\begin{eqnarray*}
A(x):= t \, \varphi'_R(|x|)
 \left( \begin{array}{cccc}
0      & 0      & \ldots  & 0  \\
0      & 0      & \ldots  & 0  \\
\vdots & \vdots & \ddots   & \vdots \\
\frac{x_1}{|x|} & \frac{x_2}{|x|} &  \ldots  & \frac{x_n}{|x|}
\end{array} \right)
\end{eqnarray*}
and~$\|A\|\leq t\left|\varphi'_R(|x|)\right|\ll 1$.
Moreover,
$$ D_y x= \left(I+A\right)^{-1}= I-\frac{1}{1+trA}A. $$
Furthermore, the following relations hold:
$$ \nabla_y u^+_{R,t}=\nabla_x u\, D_y x, \qquad \nabla_y v^+_{R,t}=\nabla_x v\, D_y x \quad
{\mbox{ and }} \quad dy=\left(1+trA\right)dx. $$
Hence, we have
\begin{eqnarray*}
&& \int_{\Omega\cap B_R} F\left(\nabla_y u^+_{R,t},\nabla_y v^+_{R,t},u^+_{R,t},v^+_{R,t},y'\right) dy \\
&=& \int_{\Omega\cap B_R} F\left(\nabla_x u\left(I-\frac{1}{1+trA}A\right), \nabla_x v\left(I-\frac{1}{1+trA}A\right),u,v,x'\right)\left(1+trA\right)dx.
\end{eqnarray*}
Now we use the hypothesis~\eqref{F2} for~$F$ to bound the right hand side from above:
more precisely, since~$\left|(p_1 A)\right|\leq|p_1\cdot e_n|/4$
and~$\left|(p_2 A)\right|\leq|p_2\cdot e_n|/4$, we have that
\begin{eqnarray*}
&&F\left(p_1\left(I-\frac{1}{1+trA}A\right),p_2\left(I-\frac{1}{1+trA}A\right),z_1,z_2,x'\right) \left(1+trA\right) \\
&\leq& F(p_1,p_2,z_1,z_2,x')\left(1+trA\right) \\
&&\quad -F_{p_1}(p_1,p_2,z_1,z_2,x')\cdot(p_1 A) - F_{p_2}(p_1,p_2,z_1,z_2,x')\cdot(p_2 A) \\
&&\quad + C\Big[\left|F_{p_1p_1}(p_1,p_2,z_1,z_2,x')\right||p_1 A|^2 + \left|F_{p_2p_2}(p_1,p_2,z_1,z_2,x')\right||p_2 A|^2 \\
&&\qquad\qquad + 2\left|F_{p_1p_2}(p_1,p_2,z_1,z_2,x')\right||p_1 A||p_2 A|\Big].
\end{eqnarray*}

If we consider~$u^-_{R,t}$ and~$v^-_{R,t}$,
we can write the same inequality with~$A$ replaced by~$-A$.
Therefore, we obtain
\begin{equation}\begin{split}\label{333}
&\mathcal E_R(u^+_{R,t},v^+_{R,t})+\mathcal E_R(u^-_{R,t},v^-_{R,t})-2\mathcal E_R(u,v) \\
\leq&\, C\int_{\Omega\cap B_R}\Big[\left|F_{p_1p_1}\right||\nabla u|^2|A|^2 + \left|F_{p_2p_2}\right||\nabla v|^2|A|^2 + 2\left|F_{p_1p_2}\right||\nabla u||\nabla v||A|^2\Big] dx  \\
\leq&\, C\frac{t^2}{(\log R)^2}\int_{\Omega\cap\left(B_R\setminus B_{\sqrt R}\right)} \frac{|F_{p_1p_1}||\nabla u|^2 +|F_{p_2p_2}||\nabla v|^2 + |F_{p_1p_2}||\nabla u||\nabla v|}{|x|^2}\, dx,
\end{split}\end{equation}
where the derivatives of~$F$ are evaluated at~$(\nabla u,\nabla v,u,v,x')$.
Now we set
$$ h(r):=\int_{\Omega\cap B_r}\left(|F_{p_1p_1}||\nabla u|^2 +|F_{p_2p_2}||\nabla v|^2 + |F_{p_1p_2}||\nabla u||\nabla v|\right) dx. $$
Recalling~\eqref{growth}, we have that~$h(r)\leq Cr^2$. Also, we see that
\begin{equation}\label{polarcord}
\int_{\sqrt R}^R \frac{h'(r)}{r^2}\, dr\leq
\frac{h(R)}{R^2} + 2\int_{\sqrt R}^R \frac{h(r)}{r^3}\, dr\leq C\log R.
\end{equation}
Passing to polar coordinates in the last integral in~\eqref{333} and using \eqref{polarcord} we get
$$
\limsup_{R\rightarrow +\infty}\sup_{t\in(0,\sqrt{R}/4)} \frac{\mathcal E_R(u^+_{R,t},v^+_{R,t})+\mathcal E_R(u^-_{R,t},v^-_{R,t})-2\mathcal E_R(u,v)}{t^2}\leq 0.
$$

\begin{rem}
We point out that in a similar way, by perturbing only one of the element
of the couple $(u,v)$, one can obtain
\begin{equation}\label{60}
\limsup_{R\rightarrow +\infty}\sup_{t\in(0,\sqrt{R}/4)} \frac{\mathcal E_R(u^+_{R,t},v)+\mathcal E_R(u^-_{R,t},v)-2\mathcal E_R(u,v)}{t^2}\leq 0
\end{equation}
and
\begin{equation}\label{60bis}
\limsup_{R\rightarrow +\infty}\sup_{t\in(0,\sqrt{R}/4)} \frac{\mathcal E_R(u,v^+_{R,t})+\mathcal E_R(u,v^-_{R,t})-2\mathcal E_R(u,v)}{t^2}\leq 0.
\end{equation}
\end{rem}

We conclude this section recalling the following integral formulas:
\begin{equation}\begin{split}\label{64}
& \mathcal E_R\left(\max\left\lbrace u^-_{R,t},u\right\rbrace,v\right) + \mathcal E_R\left(\min\left\lbrace u^-_{R,t},u\right\rbrace,v\right) = \mathcal E_R(u^-_{R,t},v) + \mathcal E_R(u,v), \\
& \mathcal E_R\left(u,\max\left\lbrace v^-_{R,t},v\right\rbrace\right) + \mathcal E_R\left(u, \min\left\lbrace v^-_{R,t},v\right\rbrace\right) = \mathcal E_R(u,v^-_{R,t}) + \mathcal E_R(u,v).
\end{split}\end{equation}

\section{Proof of Theorem~\ref{T1}}\label{sec:proof}

We recall the notation introduced at the beginning of Section \ref{sec:infi},
and we observe that, since~$(u,v)$ is an~$e_n$-minimizer,
$$ \mathcal E_R(u,v)\leq \mathcal E_R(u^+_{R,t},v). $$
From this and~\eqref{60} we obtain that
\begin{equation}\label{61}
\lim_{R\rightarrow +\infty}\mathcal E_R(u^-_{R,t},v)-\mathcal E_R(u,v)=0,
\end{equation}
at $t$ fixed.

Using again the minimality of~$(u,v)$, we have that
$$ \mathcal E_R(u,v)\leq\mathcal E_R\left(\min\left\lbrace u^-_{R,t},u\right\rbrace,v\right), $$
which, together with the first relation in \eqref{64}, implies that
\begin{equation}\label{62}
\mathcal E_R\left(\max\left\lbrace u^-_{R,t},u\right\rbrace,v\right)- \mathcal E_R(u^-_{R,t},v) = \mathcal E_R(u,v) -\mathcal E_R\left(\min\left\lbrace u^-_{R,t},u\right\rbrace,v\right)\leq 0.
\end{equation}
Putting together~\eqref{61} and~\eqref{62} we obtain
\begin{equation}\label{63}
\lim_{R\rightarrow +\infty} \mathcal E_R\left(\max\left\lbrace u^-_{R,t},u\right\rbrace,v\right) -\mathcal E_R(u,v) =0.
\end{equation}

We set
\begin{equation}\label{66}
f_{R,t}:=\max\left\lbrace u^-_{R,t},u\right\rbrace
\end{equation}
and we observe that
$$ f_{R,t}=\max\left\lbrace u(x),u(x+te_n)\right\rbrace $$
and~$f_{R,t}\in D^t_R(u)$.

Now, we argue by contradiction, assuming that~$u\in C^1(\Omega)$
is not monotone on a line in the direction~$e_n$.
This implies that we can take~$t>0$ in such a way that~$u(x)$ and~$u(x+te_n)$
satisfy the hypotheses of Lemma~\ref{lem1}, say at some point~$x_0\in\Omega$
(see Remark~$4.2$ in~\cite{SaVa}).
Therefore, we have that~$g_{R,t}:=(f_{R,t},v)$ is not an~$e_n$-minimizer for~$\mathcal E$.
Hence, in a neighborhood of~$x_0$, we can perturb~$g_{R,t}$
into~$\tilde{g}_{R,t}$ in such a way that
$$ \mathcal E_R(\tilde{g}_{R,t})\leq \mathcal E_R(g_{R,t})-c $$
for some~$c>0$ which depends only on~$(u,v)$.
From the last inequality and~\eqref{63} we reach a contradiction
with the minimality of~$(u,v)$ as~$R\rightarrow +\infty$.

If we assume that~$v\in C^1(\Omega)$ is not monotone on a line in the~$e_n$-direction,
we get again a contradiction with the fact that~$(u,v)$ is an~$e_n$-minimizer.
Indeed we can repeat the same argument as above, using~\eqref{60bis} and
the second integral formula in~\eqref{64}.
This concludes the proof of Theorem~\ref{T1}.

\section{Proof of Theorem~\ref{T2}}\label{sec:proof2}

We recall the notation introduced at the beginning of Section \ref{sec:infi}.
From~\eqref{60} we deduce that, for~$\epsilon>0$, we can take~$R$ large such that
\begin{equation}\label{67}
\mathcal E_R(u^+_{R,t},v) +\mathcal E_R(u^-_{R,t},v) -2\mathcal E_R(u,v)\leq \epsilon\, t^2.
\end{equation}
Moreover, we know that~$(u,v)$ is~$\left\lbrace e_n,e_{n+1}\right\rbrace$-stable,
and so
\begin{equation}\label{68}
\mathcal E_R(w_1,v)-\mathcal E_R(u,v)\geq -\epsilon\, t^2 \qquad \mbox{for any } w_1\in\mathcal D^t_R(u),
\end{equation}
for every~$t$ sufficiently small (see Definition~\ref{stab}).

Now, recalling the definition of~$f_{R,t}$ in~\eqref{66},
we use the first integral formula in~\eqref{64} to obtain that
\begin{eqnarray*}
&& \mathcal E_R(f_{R,t},v) -\mathcal E_R(u,v) \\
&=& \mathcal E_R(u^-_{R,t},v)- \mathcal E_R\left(\min\left\lbrace u^-_{R,t},u\right\rbrace,v\right) \\
&=& \mathcal E_R(u^-_{R,t},v) + \mathcal E_R(u^+_{R,t},v) -2\mathcal E_R(u,v)
+ \mathcal E_R(u,v) \\&&\qquad - \mathcal E_R\left(\min\left\lbrace u^-_{R,t},u\right\rbrace,v\right)
+ \mathcal E_R(u,v) -\mathcal E_R(u^+_{R,t},v) \\
&\leq& 3\epsilon\, t^2,
\end{eqnarray*}
where we have used~\eqref{67} and~\eqref{68}.

By contradiction, we assume that~$u_n$ changes sign in~$\Omega$.
Then, by Lemma \ref{beH}, there exists a point~$x_0\in\Omega$ such that
the hypotheses of Lemma~\ref{lem2}
are satisfied in a neighborhood of~$x_0$.
Hence, we have that we can perturb~$f_{R,t}$
into~$\tilde{f}_{R,t}$ near~$x_0$ in such a way that
$$ \mathcal E_R(\tilde{f}_{R,t},v)\leq\mathcal E_R(f_{R,t},v)-c\, t^2, $$
where~$\tilde{f}_{R,t}\in\mathcal D^{Ct}_R(u)$, for some~$c$, $C>0$
which depend only on~$u$.
Then, we obtain
$$ \mathcal E_R(\tilde{f}_{R,t},v)\leq \mathcal E_R(u,v) +(3\epsilon-c)t^2, $$
which gives a contradiction with the stability inequality~\eqref{stabineq}
if we take~$\epsilon\ll c$.

If we assume that~$v_n$ changes sign in~$\Omega$,
we reason in a similar way, using~\eqref{60bis},
the second integral formula in~\eqref{64} and Remark~\ref{remlem}
to reach the same contradiction.
Therefore, either~$u_n\geq0$ or~$u_n\leq0$ and either~$v_n\geq0$ or~$v_n\leq0$ in~$\Omega$.
This completes the proof of Theorem~\ref{T2}.

\section{Some applications}\label{applications}

\subsection{Two-stated mixture of Bose-Einstein condensate}
Here we consider the following system
\begin{eqnarray}
\left\{
\begin{array}{r c l}
\Delta u=uv^2,\\
\Delta v= vu^2,\\
u,v>0.
\end{array}
\right. \label{boseeinstein}
\end{eqnarray}
As already mentioned in the Introduction, the above system arises in the analysis
of phase separation phenomena in binary mixtures of Bose-Einstein condensates with multiple states.

The energy associated to \eqref{boseeinstein} is the following:
$$ \mathcal{E}(u,v):= \frac12 \int_{\R^n} |\nabla u|^2 +
|\nabla v|^2 + u^2v^2.$$
Under a suitable growth condition, we show here that stable solutions of \eqref{boseeinstein}
are one-dimensional,
i.e. there exist $\overline u$, $\overline v:\R\to\R$ and $\omega_u$, $\omega_v\in\S^{n-1}$ such that
$(u(x),v(x))=\big(\overline u(\omega_u\cdot x), \overline v(\omega_v\cdot x)\big)$.
Precisely:

\begin{prop} \label{prop:bose}
Let $u,v\in C^2(\R^n)$.
Suppose that $(u,v)$ is a stable solution to \eqref{boseeinstein}, and that the following
growth condition holds true: there exists a constant $C>0$ such that, for $R$ sufficiently large,
\begin{equation}\label{gr2}
\int_{B_R} |\nabla u|^2+|\nabla v|^2 \le C R^2.
\end{equation}

Then $(u,v)$ possesses one-dimensional symmetry.
\end{prop}

\begin{proof}
We define
$$ F(p_1,p_2,z_1,z_2,x'):=\frac12\left( |p_1|^2+|p_2|^2+z_1^2 z_2^2\right).$$
Notice that $F$ is convex in $p_1$ and $p_2$ and satisfies \eqref{F2}.
Moreover $F_{p_1p_1}=1=F_{p_2p_2}$ and $F_{p_1p_2}=0$, therefore the growth condition \eqref{growth} is ensured by \eqref{gr2}.
Then, we can apply Theorem \ref{T2} (recall also Lemma \ref{sens})
and we obtain that $(u,v)$ is one-dimensional.
\end{proof}

As a corollary, we obtain the one-dimensional symmetry in dimension $2$ for stable solutions to \eqref{boseeinstein}
which have linear growth at infinity (see \cite{BSWW}).

\begin{corollary}
Let $n=2$ and let $u,v\in C^2(\R^2)$.
Suppose that $(u,v)$ is a stable solution to \eqref{boseeinstein},
and that the following growth condition holds true:
there exists a constant $C>0$ such that
\begin{equation}\label{gr3}
|u(x)|+|v(x)|\le C(1+|x|).
\end{equation}

Then $(u,v)$ possesses one-dimensional symmetry.
\end{corollary}

\begin{proof}
We observe that, by Theorem 5.1 in \cite{Wa}, if $(u,v)$ satisfies \eqref{gr3}
then $|\nabla u|$ and $|\nabla v|$ are bounded in $\R^2$. Therefore,
condition \eqref{gr2} is trivially satisfied,
and so we get the desired result thanks to Proposition \ref{prop:bose}.
\end{proof}

\subsection{General systems with $p$-Laplacian type operators}
Theorem \ref{T2} actually applies to a broader class
of operators and nonlinearities.
Indeed, with the notation introduced in~\cite{Di} (see in particular pages 3474-3475 there),
one can consider
\begin{eqnarray}
\left\{
\begin{array}{r c l}
\mathrm{div} \big(a(|\nabla u|)\nabla u\big)=\tilde F_1(u,v),\\
\mathrm{div} \big(b(|\nabla v|)\nabla v\big)=\tilde F_2(u,v),
\end{array}
\right. \label{ppp1}
\end{eqnarray}
where $\tilde F$ is a $C^{1,1}_{loc}$ function on~$\R^2$
(it corresponds to the function~$F$ introduced in~\cite{Di},
here denoted as~$\tilde F$ to avoid confusion), and $\tilde{F}_1$, $\tilde{F}_2$
denote the derivatives of~$\tilde F$ with respect to the first and the second variable respectively.

Then, we have the following:

\begin{prop} \label{prop:plapl}
Let $u\in C^1(\R^n)\cap C^2(\{\nabla u\neq 0\})$ and $v\in C^1(\R^2)\cap C^2(\{\nabla v\neq 0\})$.
Suppose that $(u,v)$ is a stable solution to~\eqref{ppp1}, and that
conditions~$(B1)$ and~$(B2)$ in~\cite{FSV} are satisfied for $a$, $b$, $A$ and $B$.

Also, assume that the following growth condition holds true:
there exists a constant $C>0$ such that, for $R$ sufficiently large,
\begin{equation}\label{ppp2}
\int_{B_R}\Lambda_2(|\nabla u|) +\Gamma_2(|\nabla v|)\le CR^2.
\end{equation}
where $\Lambda_2$ and $\Gamma_2$ are as in \cite{Di}.
Then $(u,v)$ possesses one-dimensional symmetry.
\end{prop}

\begin{proof}
We define
$$ F(p_1,p_2,z_1,z_2,x'):=\Lambda_2(|p_1|) +\Gamma_2(|p_2|) + \tilde F(z_1,z_2),$$
and we verify that
$$ \int_{\R^n} F(\nabla u,\nabla v, u, v, x')\, dx $$
satisfies the hypotheses needed to apply Theorem~\ref{T2}. Observe, first, that being stable for the above energy is the same a being stable for \eqref{ppp1} (see Definition 1.2 in \cite{Di}). Now, recalling the notations used in~\cite{Di}, we derive
\begin{eqnarray*}
&& F_{(p_1)_i} (p_1,p_2,z_1,z_2,x') = \lambda_2(|p_1|)_{(p_1)_i} = a(|p_1|)(p_1)_i,\\
&& F_{(p_2)_i} (p_1,p_2,z_1,z_2,x') = \gamma_2(|p_2|)_{(p_2)_i} = b(|p_2|)(p_2)_i,\\
&& F_{(p_1)_i(p_1)_j} (p_1,p_2,z_1,z_2,x') = a(|p_1|)\delta_{ij} + a'(|p_1|)\frac{(p_1)_i(p_1)_j}{|p_1|}
= A_{ij}(p_1),\\
&& F_{(p_2)_i(p_2)_j} (p_1,p_2,z_1,z_2,x') = b(|p_2|)\delta_{ij} + b'(|p_2|)\frac{(p_2)_i(p_2)_j}{|p_2|}
= B_{ij}(p_2)\\
{\mbox{and }} && F_{(p_1)_i(p_2)_j} (p_1,p_2,z_1,z_2,x') = 0.
\end{eqnarray*}
Therefore, Lemma~2.1 in~\cite{Di} implies the desired convexity properties on~$F$.
Moreover, it also gives that~$|F_{p_1p_1}(p_1,p_2,z_1,z_2,x')|$ and
$|F_{p_2p_2}(p_1,p_2,z_1,z_2,x')|$ are bounded from above and
below by $\left(\lambda_1(|p_1|)+\lambda_2(|p_1|)\right)$ and
$\left(\gamma_1(|p_2|)+\gamma_2(|p_2|)\right)$, respectively, up to multiplicative constants.

Furthermore, since conditions~$(B1)$ and~$(B2)$ in~\cite{FSV} are satisfied for $a$, $b$, $A$ and $B$,
one can use Lemma~4.2 there to see that, if $|p_1|$, $|p_2|\le M$, then
\begin{eqnarray*}
&& \lambda_1(|p_1|)\le C_M\lambda_2(|p_1|), \qquad \lambda_2(|p_1+q_1|)\le C_M\lambda_2(|p_1|), \\
&& \gamma_1(|p_2|)\le C_M\gamma_2(|p_2|) \quad {\mbox{ and }}\quad
\gamma_2(|p_2+q_2|)\le C_M\gamma_2(|p_2|),
\end{eqnarray*}
if $|q_1|\le |p_1|/2$ and $|q_2|\le |p_2|/2$, for some $C_M>0$. As a consequence,
\begin{equation}\begin{split}\label{9.18}
& |F_{p_1p_1}(p_1,p_2,z_1,z_2,x')|\le C_M\lambda_2(|p_1|)\\
{\mbox{ and }} & |F_{p_2p_2}(p_1,p_2,z_1,z_2,x')|\le C_M\gamma_2(|p_2|),
\end{split}\end{equation}
up to rename~$C_M$. Therefore, recalling also formulas (1.4)-(1.7) in~\cite{Di}, we get
\begin{eqnarray*}
&& |F_{p_1p_1}(p_1+q_1,p_2,z_1,z_2,x')|\le C_M \lambda_2(|p_1+q_1|)\\
&&\qquad\qquad \le C_M \lambda_2(|p_1|)\le C_M
|F_{p_1p_1}(p_1,p_2,z_1,z_2,x')| \\
{\mbox{and }} && |F_{p_2p_2}(p_1,p_2+q_2,z_1,z_2,x')|\le C_M \gamma_2(|p_2+q_2|)\\
&&\qquad\qquad\le
C_M \gamma_2(|p_2|)\le C_M |F_{p_2p_2}(p_1,p_2,z_1,z_2,x')|,
\end{eqnarray*}
for any $2|q_1|\le|p_1|\le M$ and $2|q_2|\le|p_2|\le M$. This means that~\eqref{F2} is satisfied.

It remains to check the growth condition~\eqref{growth}.
For this, notice that~\eqref{9.18} and formula~$(4.13)$ in~\cite{FSV} give that
\begin{eqnarray*}
&& |F_{p_1p_1}(p_1,p_2,z_1,z_2,x')|\, |p_1|^2 + |F_{p_2p_2}(p_1,p_2,z_1,z_2,x')|\, |p_2|^2 \\
&&\qquad \le
C_M \left(\lambda_2(|p_1|)|p_1|^2 + \gamma_2(|p_2|)|p_2|^2\right) \\
&&\qquad = C_M \left( a(|p_1|)|p_1|^2 + b(|p_2|)|p_2|^2\right)
\le C_M\left(\Lambda_2(|p_1|)+ \Gamma_2(|p_2|)\right).
\end{eqnarray*}
This and~\eqref{ppp2} imply that
\begin{eqnarray*}
&& \int_{B_R}  |F_{p_1p_1}(\nabla u,\nabla v,u,v,x')|\, |\nabla u|^2 +
|F_{p_2p_2}(\nabla u,\nabla v,u,v,x')|\, |\nabla v|^2\, dx \\
&&\qquad \le C_M \int_{B_R}
\left( \Lambda_2(|\nabla u|)+ \Gamma_2(|\nabla v|)\right)\, dx \le C_M R^2, \end{eqnarray*}
up to rename~$C_M$. This shows~\eqref{growth} holds true.

Hence, we can apply Theorem~\ref{T2}, thus obtaining the desired monotonicity property.
\end{proof}

As paradigmatic examples in Proposition~\ref{prop:plapl} one can take the $p$-Laplacian,
with~$p\in(1,+\infty)$ if~$\{\nabla u=0\}=\varnothing$ and~$p\in[2,+\infty)$
if~$\{\nabla u=0\}\neq\varnothing$ (in this case~$a(t)=t^{p-2}$),
or the mean curvature operator (in this case~$a(t)=(1+t^2)^{-1/2}$).
We stress on the fact that one can also take different operators~$a$ and~$b$ satisfying
the hypotheses (for instance, one can take~$a$ to be the $p$-Laplacian and~$b$
an operator of mean curvature type).

Furthermore, we observe that condition~\eqref{ppp2} is satisfied,
for instance, when~$n=2$ and~$\nabla u$ and~$\nabla v$ are bounded
(thanks to the hypotheses on~$a$ and~$b$, see page 3474 in~\cite{Di}),
This means that we recover Theorem~7.1 of~\cite{Di} for stable solutions,
without requiring conditions on the sign of~$\tilde F_{12}$.

\subsection{Systems involving the fractional Laplacian}
The general setting of our results allows us to treat also the case of nonlocal systems of equations, i.e.
\begin{eqnarray}
\left\{
\begin{array}{r c}
(-\Delta)^{s_1}u=\tilde F_1(u,v),\\
(-\Delta)^{s_2}v=\tilde F_2(u,v),
\end{array}
\right. \label{fractional}
\end{eqnarray}
where~$s_1$, $s_2\in(0,1)$, $\tilde F$ is a $C^{1,1}_{loc}$ function on~$\R^2$ and $\tilde F_1$ and $\tilde F_2$
denote the derivatives of~$\tilde F$ with respect to the first and the second variable respectively.

As a matter of fact, as in Remark~$2.12$ of~\cite{SaVa},
we observe that one can generalize the functional considered in~\eqref{quitegen} to the following
functional
\begin{equation}\label{nuovofunct}
\int_{\Omega} F(\nabla u,\nabla v,u,v,x')\,dx +\int_{\partial\Omega} G(u,v,x')\, d\mathcal{H}^{n-1},
\end{equation}
where $G$ satisfies the same regularity assumptions as~$F$.

Furthermore, the growth condition in~\eqref{growth} can be weakened in the following way.
We define~$\ell_0(R):=R$ and
$$ \ell_k(R):=\log \big(\ell_{k-1}(R)\big)= \underbrace{\log\circ\ldots\circ\log}_{ {\mbox{$k$ times}}} R \quad {\mbox{ for any }}
k\in\N, k\ge1.$$
We also set
$$ \pi_k(R):=\prod_{j=0}^k \ell_j(R).$$
Then, one can require that, for some~$k\in\N$,
\begin{equation}\begin{split}\label{picond}
&\int_{\Omega\cap B_R}|F_{p_1p_1}(\nabla u,\nabla v,u,v,x')||\nabla u|^2 +
|F_{p_2p_2}(\nabla u,\nabla v,u,v,x')||\nabla v|^2 \\
&\qquad\qquad\qquad\qquad\qquad + |F_{p_1p_2}(\nabla u,\nabla v,u,v,x')||\nabla u||\nabla v|\, dx
\leq CR\pi_k(R),
\end{split}\end{equation}
instead of~\eqref{growth}. Notice that if~$k=0$ then~\eqref{picond}
corresponds to~\eqref{growth}. See Remark~$2.13$ and Section~9 in~\cite{SaVa} for the proof
and further discussion on this fact.

We also recall that, using the extension result in~\cite{CS},
one can localize problem~\eqref{fractional} by adding one variable,
see e.g. formula~$(1.7)$ in~\cite{pinamonti}.
We will denote by~$U$ and~$V$
the extension functions of~$u$ and~$v$, respectively (see e.g. formulas~$(2.3)$
and~$(2.4)$ in~\cite{pinamonti}).

So one has to deal with an energy functional
like~\eqref{nuovofunct} with
\begin{equation}\begin{split}\label{asdqw98}
& \Omega:=(0,+\infty)\times\R^n, \quad V:=(0,+\infty)\times\R^{n-1},\\
& F(p_1,p_2,z_1,z_2,x'):=x_1^{1-2s_1}|p_1|^{2} + x_1^{1-2s_2}|p_2|^2\\
{\mbox{ and }} &G(z_1,z_2,x'):= \tilde F(z_1,z_2)
\end{split}\end{equation}
(notice that in this application one has to replace~$n$ with~$n+1$ to apply Theorem~\ref{T2}).

With this, we can prove the following:

\begin{prop} \label{prop:fraclapl}
Let $u$, $v\in C^2(\R^{n})$.
Suppose that $(u,v)$ is a stable solution to~\eqref{fractional}.
Also, assume that the following growth condition holds true in the extension:
there exist a constant $C>0$ and~$k\in\N$ such that, for $R$ sufficiently large,
\begin{equation}\label{poiuyt234}
\int_{B_R}x_1^{1-2s_1}|\nabla U|^2 +x_1^{1-2s_2}|\nabla V|^2\,dx_1\,\cdots dx_{n+1} \le CR\pi_k(R).
\end{equation}

Then $(u,v)$ possesses one-dimensional symmetry.
\end{prop}

\begin{proof}
With the notation introduced in~\eqref{asdqw98}, we observe that the thesis
simply follows from~\eqref{poiuyt234} (that ensures the growth
condition~\eqref{picond}) and Theorem~\ref{T2}.
\end{proof}

We remark that~\eqref{poiuyt234} is a reasonable energy growth condition,
since it is satisfied for instance in the case of a single equation
when~$n=2$ for any~$s\in(0,1)$ and when~$n=3$ for~$s\in(1/2,1)$,
see~\cite{CC1,CC2}.
Moreover, if~$s_1=s_2=1/2$, it can be checked  as in~\cite{CC1}
with suitable modifications under an additional assumption on the bound
of~$\nabla U$ and~$\nabla V$ (see in particular formula~$(1.16)$ and Section~$4$ in~\cite{CC1}).

We finally remark that, differently from~\cite{pinamonti},
we do not need here any sign assumption on~$\tilde F$.

\section*{Acknowledgments} The authors want to thank \emph{Enrico Valdinoci}
for very helpful discussions and comments.

This work has been supported by Alexander von Humboldt Foundation and
ERC grant 277749 "EPSILON Elliptic PDE's and Symmetry of Interfaces and Layers for Odd Nonlinearities".

\vspace{2mm}

\end{document}